\documentclass[reqno,a4paper,12pt]{amsart}

\numberwithin{equation}{section}

\usepackage[utf8]{inputenc}
\usepackage{amsmath, amsthm, amssymb, amscd, accents, bm, amsfonts}
\usepackage{mathtools}
\usepackage{mathrsfs,dsfont}
\usepackage{datetime}
\usepackage{hyperref}
\usepackage{colonequals}

\usepackage[sort,nocompress]{cite}
\mathtoolsset{showonlyrefs}

\newtheorem{definition}{Definition}[section]
\newtheorem{theorem}[definition]{Theorem}
\newtheorem{lemma}[definition]{Lemma}

\def\N{{\mathbb N}}
\def\Z{{\mathbb Z}}
\def\R{{\mathbb R}}
\def\T{{\mathbb T}}
\def\C{{\mathbb C}}
\def\Q{{\mathbb Q}}

\newcommand{\lt}{{L^2(\R)}}

\newcommand{\ft}{{\mathcal{F}}}
\newcommand{\supp}{\operatorname{supp}}
\newcommand\1{\mathds{1}}

\begin{document}

\title[Compactly Supported Gabor Orthonormal Bases]{Compactly Supported Gabor Orthonormal Bases}

\author[Lukas Liehr]{Lukas Liehr}
\address{Department of Mathematics, Bar-Ilan University, Ramat-Gan 5290002, Israel}
\email{lukas.liehr@biu.ac.il}

\date{\today}

\begin{abstract}
We characterize all lattices $\Lambda \subset \R^2$ and all compactly supported functions $g \in L^2(\R)$ for which the Gabor system $$\left \{ e^{2\pi i s x} g(x-t) : (t,s) \in \Lambda \right \}$$ forms an orthonormal basis for $L^2(\R)$. The characterization is given in geometric terms through translation tilings and discreteness properties of lattice projections. In particular, this resolves a conjecture of Han and Wang on the non-existence of Gabor bases along specific irrational lattices. Finally, we construct Gabor bases that cannot be realized by any product set, answering a problem of Iosevich and Mayeli.
\end{abstract}

\maketitle

\section{Introduction and main results}

\subsection{}
Let $g \in \lt$ and let $\Lambda \subset \R^2$ be a countable set such as a full-rank lattice, i.e., $\Lambda = A\Z^2 := \{ Az : z \in \Z^2 \}$ for an invertible matrix $A \in \R^{2 \times 2}$. The Gabor system of $g$ along $\Lambda$ is given by the system $\mathbf G(g,\Lambda) \subset \lt$ which is defined by
$$
\mathbf G(g,\Lambda) := \left \{ e^{2\pi i s x} g(x-t) : (t,s) \in \Lambda 
 \right \}.
$$
The function $g$ is called the window-function of the system $\mathbf G(g,\Lambda)$.
The most classical example of a Gabor system with compactly supported window function that forms an orthonormal basis for $\lt$ is given by $\mathbf G(\1_{[0,1]},\Z^2)$ where $\1_S$ denotes the characteristic function of a measurable set $S$. The property of $\mathbf G(\1_{[0,1]},\Z^2)$ being an orthonormal basis follows directly from the fact that the exponential system $\{ e^{2\pi i n x} : n \in \Z \}$ is an orthonormal basis for $L^2(0,1)$. The question of which other compactly supported functions $g$ have the property that $\mathbf G(g,\Lambda)$ is an orthonormal basis was investigated by various authors \cite{L05,AAK20,DL22,HW01,grepstad2026bounded} and is sometimes referred to as the Fuglede--Gabor problem, in analogy with the classical Fuglede problem on bases of exponentials \cite{fuglede1974commuting,tao2004fuglede,lev2022fuglede,greenfeld2017fuglede}.
This question has been studied from several different perspectives in the literature. In particular, the problem has been investigated in connection with tiling and rigidity properties of the set $\Lambda$ \cite{HW01,HW04,LW03,LM19,AAK20,DL22,GLW15,grepstad2026bounded}, as well as in higher-dimensional and convex-geometric settings \cite{GLW15,BO18,CL18,IM18}. Related questions have also been considered in discrete settings \cite{Z21}, over locally compact Abelian groups and fields \cite{IKL21,N24}, and in the context of Schauder bases \cite{BM21,HP06,LT26}.

\subsection{}
In the present paper we investigate the following problem: for what lattices $\Lambda$ and compactly supported functions $g$ does $\mathbf G(g,\Lambda)$ form an orthonormal basis? While the existence of \emph{some} function $g \in \lt$ yielding an orthonormal basis is settled by a result of Bekka \cite{bekka}, the existence of a compactly supported function constitutes an open problem. This problem is further motivated by a conjecture of Han and Wang \cite{HW01,HW04,LM19} as well as a question raised by Iosevich and Mayeli \cite{IM18} (see the discussion below for the precise conjecture and question). In the following we give a complete answer to this problem. To do so, we denote for a set $Q \subset \R^2$ the projection onto the first coordinate by $\pi_1(Q)$,
$$
\pi_1(Q) = \{ a \in \R : \exists b \in \R \ \, \text{such that} \ \, (a,b) \in Q \}.
$$
Moreover, $D(\Lambda)$ denotes the density of a lattice $\Lambda$, i.e., the reciprocal of the volume of a fundamental domain of $\Lambda$. We start by characterizing all lattices for which there exists a compactly supported window function generating a Gabor orthonormal basis.

\begin{theorem}\label{thm:main1}
    For every lattice $\Lambda \subset \R^2$, the following are equivalent:
    \begin{enumerate}
        \item there exists a compactly supported $g \in \lt$ such that $\mathbf{G}(g,\Lambda)$ is an orthonormal basis for $\lt$
        \item $D(\Lambda)=1$ and $\pi_1(\Lambda) = a \Z$ for some $a>0$.
    \end{enumerate}
\end{theorem}

The previous theorem yields a proof of the Han-Wang conjecture. Han and Wang conjectured (implicitly in \cite{HW01} and explicitly in \cite[p. 186]{HW04}; see also \cite[p. 3095]{LM19}) that if $\alpha \in \R \setminus \Q$, then for the lattice $\Lambda_\alpha \subset \R^2$ defined by
    \begin{equation}\label{eq:han_wang_matrix}
        \Lambda_\alpha = A\Z^2, \quad
    A = \begin{pmatrix}
        1 & \alpha \\ 0 & 1
    \end{pmatrix},
    \end{equation}
    there exists \emph{no} compactly supported $g \in \lt$ such that $\mathbf G(g,\Lambda_\alpha)$ is an orthonormal basis. Theorem \ref{thm:main1} gives an immediate affirmation of this conjecture: since $\alpha$ is irrational it follows that $\pi_1(\Lambda) = \Z + \alpha \Z$ is dense in $\R$. Hence, there exists no $a > 0$ such that $\pi_1(\Lambda) = a\Z$.
    
\subsection{}
Now suppose that $\Lambda$ is a lattice with $D(\Lambda)=1$ and $\pi_1(\Lambda) = a \Z$ for some $a>0$. According to Theorem \ref{thm:main1} we can find a compactly supported $g$ such that $\mathbf{G}(g,\Lambda)$ is an orthonormal basis for $\lt$. The next theorem gives a characterization of all such $g$ that yield a basis. Hence, in combination with Theorem \ref{thm:main1}, we obtain a complete picture of when a basis exists and what windows generate the basis.

\begin{theorem}\label{thm:main2}
    Let $\Lambda \subset \R^2$ be a lattice with density $D(\Lambda) = 1$ and $\pi_1(\Lambda) = a\Z$ for some $a>0$. Further, let $g \in \lt$ be compactly supported. Then the following are equivalent.
    \begin{enumerate}
        \item $\mathbf{G}(g,\Lambda)$ is an orthonormal basis
        \item $|g|=\frac{1}{\sqrt{a}} \1_\Omega$ for a bounded measurable set $\Omega \subset \R$ that tiles the real line with $a\Z$-translates.
    \end{enumerate}
\end{theorem}

\subsection{}
If $g$ is compactly supported and $\Lambda$ is a lattice with density $D(\Lambda)=1$ and projection $\pi_1(\Lambda) = a \Z$, then from the previous statements one can further obtain that if $\mathbf{G}(g,\Lambda)$ is an orthonormal basis, then also $\mathbf{G}(g,A \times B)$ is an orthonormal basis where $A$ and $B$ are the arithmetic progressions $A = a \Z$ and $B=a^{-1}\Z$. Hence, although the lattice $\Lambda$ need not itself have a product structure, we are able to find a product set $A \times B$ so that the same window $g$ generates an orthonormal basis along $A \times B$. Iosevich and Mayeli asked in \cite[Question 4.4]{IM18} whether there exist a function $g$ and a set $\Lambda$ such that $\mathbf{G}(g,\Lambda)$ is an orthonormal basis for $\lt$, while $\mathbf{G}(g,A\times B)$ fails to be an orthonormal basis for every choice of sets $A,B\subseteq\R$. The preceding discussion shows that compactly supported windows cannot provide such an example in the lattice setting. However, the next theorem highlights that by replacing the compactly supported function $g$ with a function of unbounded support, one obtains an affirmative answer. Here, we denote by $R_\theta \in \R^{2 \times 2}$ the rotation matrix with angle $\theta$.

\begin{theorem}\label{thm:main3}
    Let $\Lambda = R_{\theta}\Z^2$ with $\theta \not \in \frac \pi2 \Z$. Then there exists $g \in \lt$ such that $\mathbf{G}(g,\Lambda)$ is an orthonormal basis for $\lt$ but $\mathbf{G}(g,A \times B)$ is not an orthonormal basis for any $A,B \subseteq \R$.
\end{theorem}

\section{Proof of the Han-Wang conjecture}

This section is devoted to proving the Han-Wang conjecture, which we will use to prove Theorem \ref{thm:main1}. Defining $T_tg(x)=g(x-t)$ and $M_sg(x)=e^{2\pi i s x} g(x)$, the Han-Wang conjecture can be stated as follows: if $\alpha$ is irrational, then there does not exist a compactly supported $g \in \lt$ such that $\left\{M_nT_{m+\alpha n}g:m,n\in\mathbb Z\right\}$ is an orthonormal basis for $\lt$.

\subsection{A weighted Zak transform}

In the following we use the Fourier transform convention
$$
\ft g(\xi) = \widehat{g}(\xi) = \int_\R g(x) e^{-2\pi i x \xi} \, dx. 
$$
We start by defining a function $D$ on $\R \times \T$ that can be regarded as a weighted Zak-transform of $\widehat g$.

\begin{lemma}\label{lem:direct-fibre-form}
Let $g \in \lt$ and let $\alpha \in \R$ such that $\left\{M_nT_{m+\alpha n}g:m,n\in\mathbb Z\right\}$ is an orthonormal system in $L^2(\R)$. Define a sequence $\{ \gamma_k \}_{k \in \Z}$ by $\gamma_k=e^{\pi i\alpha k(k-1)}$. Moreover, define for a.e. $\omega\in\mathbb R$,
$$
D(\omega,\theta) = \sum_{k\in\mathbb Z} \gamma_k\widehat g(k+\omega)e^{2\pi i k\theta}
$$
where the series is interpreted as an $L^2(\T)$-Fourier series in $\theta$. Then 
\begin{equation}\label{eq:direct-fibre-covariance}
D(\omega+1,\theta) = e^{2\pi i(\alpha-\theta)}D(\omega,\theta-\alpha)
\end{equation}
for a.e. $(\omega,\theta) \in \R \times \T$.
Moreover, $|D(\omega,\theta)|=1$ for a.e. $(\omega,\theta)\in\mathbb R\times\T$.
\end{lemma}

\begin{proof}
Define the sequences $h(\omega)=\{h_k(\omega)\}_{k\in\Z}$ and $b(\omega)=\{b_k(\omega)\}_{k\in\Z}$ by
$$
h_k(\omega)=\widehat g(k+\omega), \quad b_k(\omega)=\gamma_k h_k(\omega).
$$
Since $\widehat g\in L^2(\mathbb R)$, it follows that for a.e. $\omega \in [0,1]$, the sequence $h(\omega)$ belongs to $\ell^2(\Z)$. Since $|\gamma_k|=1$, the same is true for $b(\omega)$.
For $n\in\mathbb Z$, define
$$
r_n(\omega) = \sum_{k\in\mathbb Z} b_{k-n}(\omega)\overline{b_k(\omega)}.
$$
This series is absolutely convergent for a.e. $\omega \in [0,1]$. Indeed, by the Cauchy-Schwarz inequality
$$
\sum_{k\in\mathbb Z} |b_{k-n}(\omega)||b_k(\omega)| \leq \left(\sum_{k\in\mathbb Z}|b_{k-n}(\omega)|^2\right)^{1/2} \left(\sum_{k\in\mathbb Z}|b_k(\omega)|^2\right)^{1/2} = \sum_{k\in\mathbb Z}|b_k(\omega)|^2.
$$
Since $|b_k(\omega)|=|h_k(\omega)|$, it follows from the estimate
$$
\int_0^1 |r_n(\omega)| \, d\omega \leq \int_0^1  \sum_{k\in\mathbb Z}|h_k(\omega)|^2 \, d\omega = \| g \|^2
$$
that $r_n \in L^1(0,1)$. The identity
$$
\ft M_nT_{m+\alpha n}g(\xi) = e^{-2\pi i(m+\alpha n)(\xi-n)} \ft g(\xi-n),
$$
in combination with orthonormality and Plancherel's theorem gives
$$
\delta_{m,0}\delta_{n,0} = \int_{\mathbb R} e^{-2\pi i(m+\alpha n)(\xi-n)} \widehat g(\xi-n)\overline{\widehat g(\xi)}\,d\xi .
$$
Equivalently, we have
$$
\delta_{m,0}\delta_{n,0} = \int_{\mathbb R} e^{-2\pi i(m+\alpha n)\xi} \widehat g(\xi-n)\overline{\widehat g(\xi)}\,d\xi .
$$
Periodization of the integral and using the definition of
$h(\omega)$, we get
$$
\delta_{m,0}\delta_{n,0} = \int_0^1 e^{-2\pi i m\omega} e^{-2\pi i\alpha n\omega} e^{2\pi i\alpha n^2} \sum_{k\in\mathbb Z} e^{-2\pi i\alpha nk} h_{k-n}(\omega)\overline{h_k(\omega)} \,d\omega .
$$
The elementary identity $\gamma_{k-n}\overline{\gamma_k} = e^{-2\pi i\alpha nk}e^{\pi i\alpha n(n+1)}$ shows that the series under the previous integral satisfies the relation
$$
\sum_{k\in\mathbb Z} e^{-2\pi i\alpha nk} h_{k-n}(\omega)\overline{h_k(\omega)} = e^{-\pi i\alpha n(n+1)}r_n(\omega).
$$
Consequently,
$$
\delta_{m,0}\delta_{n,0} = e^{\pi i\alpha n(n-1)} \int_0^1 e^{-2\pi i m\omega} e^{-2\pi i\alpha n\omega} r_n(\omega)\,d\omega.
$$
Since the prefactor is nonzero, we obtain
$$
\int_0^1 e^{-2\pi i m\omega} \bigl(e^{-2\pi i\alpha n\omega}r_n(\omega)\bigr) \,d\omega = \delta_{m,0}\delta_{n,0}, \quad  m\in\mathbb Z.
$$
We have shown above that $r_n \in L^1(0,1)$. Therefore the function
$$
e^{-2\pi i\alpha n\omega}r_n(\omega)
$$
belongs to $L^1(0,1)$. Hence, by uniqueness of Fourier coefficients,
$$
e^{-2\pi i\alpha n\omega}r_n(\omega)=\delta_{n,0}
$$
for a.e. $\omega\in[0,1]$. Thus, after removing one null set independent of $n$,
$$
r_n(\omega)=\delta_{n,0}, \quad n \in \mathbb Z,
$$
for a.e. $\omega\in[0,1]$.
For such an $\omega$, define
$$
D(\omega,\theta) = \sum_{k\in\mathbb Z}b_k(\omega)e^{2\pi i k\theta}
$$
as an $L^2(\T)$-Fourier series. Then $|D(\omega,\cdot)|^2\in L^1(\T)$. Its Fourier coefficients are the autocorrelations of $b(\omega)$,
$$
\int_0^1 |D(\omega,\theta)|^2e^{2\pi i n\theta}\,d\theta = \sum_{k\in\mathbb Z} b_{k-n}(\omega)\overline{b_k(\omega)} = r_n(\omega) = \delta_{n,0}.
$$
Thus all Fourier coefficients of the $L^1$-function $|D(\omega,\cdot)|^2$ agree with those of the constant function $1$. By uniqueness of Fourier coefficients, $|D(\omega,\theta)|^2=1$ for a.e. $\theta\in \T $. Consequently, $|D(\omega,\theta)|=1$ for a.e. $(\omega,\theta)\in[0,1]\times \T$.
Further, we have
$$
\begin{aligned}
D(\omega+1,\theta)
&=
\sum_{k\in\mathbb Z}
\gamma_k\widehat g(k+\omega+1)e^{2\pi i k\theta}  \\
&=
\sum_{j\in\mathbb Z}
\gamma_{j-1}\widehat g(j+\omega)e^{2\pi i(j-1)\theta}  \\
&=
e^{2\pi i(\alpha-\theta)}
\sum_{j\in\mathbb Z}
\gamma_j\widehat g(j+\omega)e^{2\pi i j(\theta-\alpha)}  \\
&=
e^{2\pi i(\alpha-\theta)}D(\omega,\theta-\alpha).
\end{aligned}
$$
Finally, since the factor $e^{2\pi i(\alpha-\theta)}$ has modulus one, this identity transfers $|D|=1$ from $[0,1]\times \T$ to every shift $(r+[0,1])\times \T$ with $r\in\mathbb Z$, implying that $|D(\omega,\theta)|=1$ for a.e. $(\omega,\theta)\in\mathbb R\times \T$.
\end{proof}

For a measurable function $F : \R \to \C$ with polynomial growth, we denote by $T_F \in \mathcal S'(\R)$ the tempered distribution induced by $F$ through the pairing
$$
        \langle T_F,\varphi\rangle
        =
        \int_{\mathbb R}F(t)\varphi(t)\,dt,
        \quad \varphi\in\mathcal S(\mathbb R),
$$
where $\mathcal{S}(\R)$ denotes the Schwartz space on the real line.
Moreover, we denote by $\ft : \mathcal S'(\R) \to \mathcal{S}'(\R)$ the distributional Fourier transform, which for $V \in \mathcal{S}'(\R)$ is defined by the action
$
\langle \ft V,\varphi \rangle = \langle V, \widehat \varphi \rangle.
$
Recall that if $U \subseteq \R$ is an open set and $V \in \mathcal{S}'(\R)$, then $V$ vanishes on $U$ if for every $\varphi \in C_c^\infty(U)$ we have $\langle V, \varphi \rangle = 0$. The support of $V$ is denoted by $\supp(V)$ and is defined as the complement of the set
$$
\bigcup \left \{ U \subseteq \R : U \, \text{open}, \ V \ \text{vanishes on} \ U \right \}.
$$

\begin{lemma}\label{lem:fibre-compact-spectrum}
Let $K\subseteq \R$ be compact, $g \in L^2(\R)$ with $\supp (g) \subseteq K$ and let $\alpha \in \R$. Suppose that $\left\{M_nT_{m+\alpha n}g : m,n \in \Z \right\}$ is orthonormal and let $D$ be defined as in Lemma \ref{lem:direct-fibre-form}. Then, for almost every $\theta \in \T$, the distributional Fourier transform of $\omega\mapsto D(\omega,\theta)$ is supported in $-K$.
\end{lemma}

\begin{proof}
For almost every $\theta$, the function $\omega\mapsto D(\omega,\theta)$ defines a tempered distribution. Call this distribution $T_\theta$.
We show that $\operatorname{supp}(\ft T_\theta)\subseteq -K$
for a.e. $\theta$. To do so, let
$
U=\mathbb R\setminus(-K).
$
It suffices to prove that, for a.e. $\theta$,
$$
\langle \ft T_\theta ,\psi\rangle = \int_{\mathbb R}\widehat\psi(\omega)D(\omega,\theta)\,d\omega =0
\qquad
\text{for every } \psi\in C_c^\infty(U).
$$
For $\psi \in C_c^\infty(U)$ define $I_\psi$ via
$$
I_\psi(\theta)
=
\int_{\mathbb R}\widehat\psi(\omega)D(\omega,\theta)\,d\omega.
$$
Since $|D|=1$ almost everywhere by Lemma \ref{lem:direct-fibre-form}
and $\widehat{\psi}\in L^1(\mathbb R)$, the integral defining $I_\psi$
is absolutely convergent for almost every $\theta$, and
$
|I_\psi(\theta)|
\leq
\|\widehat\psi\|_{L^1(\mathbb R)} .
$
Hence $I_\psi\in L^\infty(\mathbb T)\subset L^2(\mathbb T)$. Next, we determine the Fourier coefficients $\{ a_k \}_{k \in \Z}$ of $I_\psi$.
For $k\in\mathbb Z$, Fubini's theorem gives
\begin{equation}
\begin{aligned}
a_k = \int_0^1 I_\psi(\theta)e^{-2\pi i k\theta}\,d\theta
&=
\int_{\mathbb R}\widehat\psi(\omega)
\left(
\int_0^1 D(\omega,\theta)e^{-2\pi i k\theta}\,d\theta
\right)
d\omega  \\
&=
\gamma_k
\int_{\mathbb R}\widehat\psi(\omega)\widehat g(k+\omega)\,d\omega .
\end{aligned}
\end{equation}
Here we used that, for almost every $\omega$, $D(\omega,\cdot)$ is the
$L^2(\mathbb T)$-Fourier series whose $k$-th Fourier coefficient is
$\gamma_k\widehat g(k+\omega)$.
A change of variables yields
$$
a_k = \gamma_k \int_\R \widehat\psi(\omega-k) \widehat g(\omega) d \omega.
$$
and by Plancherel, we have
$$
\int_\R \widehat\psi(\omega-k) \widehat g(\omega) d \omega =
\int_\R g(x) e^{-2\pi i k x} \psi(-x) \, dx.
$$
Since $\psi \in C_c^\infty(U)$, the supports of $g$ and $\psi(-\cdot)$ are disjoint and therefore $a_k=0$. Since $k$ was arbitrary, we have $I_\psi = 0$. Hence for each \emph{fixed} $\psi \in C_c^\infty(U)$ we have
$$
\langle \widehat{T_\theta}, \psi \rangle = 0
$$
for almost every $\theta$. In other words, for every $\psi \in C_c^\infty(U)$ there exists a set $\Omega(\psi) \subseteq \T$ of full Lebesgue measure such that
$$
\langle \widehat{T_\theta}, \psi \rangle = 0, \quad \theta \in \Omega(\psi).
$$
It remains to remove the dependence of the full measure set on $\psi$.

To do so, we argue as follows:
since $U$ is open and $\sigma$-compact, choose compact sets
$$
L_j=\{x\in U: |x|\le j,\ \operatorname{dist}(x,-K)\ge 1/j\},
\qquad j\in\mathbb N,
$$
so that every compact subset of $U$ is contained in some $L_j$. For each
$j$, the space
$
C_{L_j}^\infty(\mathbb R)
:=
\{\psi\in C_c^\infty(\mathbb R):\operatorname{supp}\psi\subseteq L_j\}
$
is separable in its usual Fréchet topology. Choose a countable dense subset
$\mathcal A_j\subset C_{L_j}^\infty(\mathbb R)$, and put
$$
\mathcal A=\bigcup_{j=1}^\infty \mathcal A_j .
$$
Since countable intersections of sets of full measure are of full measure, the preceding argument shows that there exists a single
full-measure set $\Omega\subseteq\mathbb T$ such that
$$
\langle \widehat{T_\theta},\psi \rangle = 0
$$
for every $\psi\in\mathcal A$ and every $\theta\in\Omega$.

Fix $\theta\in\Omega$. Let $\psi\in C_c^\infty(U)$. Then
$\operatorname{supp}\psi\subseteq L_j$ for some $j$. Choose
$\psi_n\in\mathcal A_j$ with $\psi_n\to\psi$ in
$C_{L_j}^\infty(\mathbb R)$. Since $\ft T_\theta$ is a distribution,
it is continuous on $C_{L_j}^\infty(\mathbb R)$. Therefore
$$
\langle \ft T_\theta,\psi\rangle
=
\lim_{n\to\infty}
\langle \ft T_\theta,\psi_n\rangle
=
0.
$$
Hence $\ft T_\theta$ vanishes on $C_c^\infty(U)$. Therefore $\operatorname{supp}(\ft T_\theta) \subseteq \mathbb R\setminus U=-K$, which proves the claim.
\end{proof}

\subsection{A complex analysis argument}

We now invoke an argument from complex analysis, where we make use of the following well-known fact which is a direct consequence of the Hadamard factorization theorem. Recall that an entire function $H : \C \to \C$ is said to be of exponential type, if there exists $C,c>0$ such that
$$
|H(z)| \leq C e^{c|z|}, \quad z \in \C.
$$

\begin{lemma}\label{lem:zero-free-exp-type}
Let $H$ be an entire function of exponential type with no zeros in $\mathbb C$. Then there exist $C\in \mathbb C\setminus\{0\}$ and $a\in\mathbb C$ such that $H(z)=Ce^{az}$.
\end{lemma}

Using the previous result we obtain the following statement.

\begin{lemma}\label{lem:bounded-unimodular-compact-spectrum}
Let $F\in L^\infty(\mathbb R)$ satisfy $|F(t)|=1$ for almost every $t\in\mathbb R$. Suppose that the distributional Fourier
transform $\widehat F$ has compact support. Then there exist
$\lambda\in\supp(\widehat F)$ and $c\in\mathbb C$, with $|c|=1$, such that $F(t)=ce^{2\pi i\lambda t}$ for almost every $t\in\mathbb R$.
\end{lemma}

\begin{proof}
Let $T_F \in \mathcal S'(\R)$ be the tempered distribution induced by $F$.
Let $U = \ft T_F$ be the Fourier transform of $T_F$.
By assumption, $U$ is compactly supported. Choose $\sigma\geq 0$ such that
$
        \operatorname{supp} U\subset [-\sigma,\sigma].
$
By the Paley-Wiener-Schwartz theorem for compactly supported distributions
\cite[Theorem 7.3.1]{H83}, the inverse Fourier transform
$\mathcal F^{-1}U$ is represented on $\mathbb R$ by the restriction of an entire function
$$
        f(z):=\big\langle U(\xi),e^{2\pi i z\xi}\big\rangle,
        \quad z\in\mathbb C.
$$
Moreover, Paley-Wiener-Schwartz gives constants $C_0>0$ and $N \in \mathbb N$ such
that
$$
        |f(z)|
        \leq
        C_0(1+|z|)^N e^{2\pi\sigma|\operatorname{Im}z|},
        \quad z\in\mathbb C .
$$
In particular, $f$ is an entire function of exponential type.
Since $\mathcal F^{-1}U=T_F$, the restriction of $f$ to $\mathbb R$
represents $F$ as a distribution. Hence $F=f$ almost everywhere. Since
$|f(t)|= |F(t)| = 1$ for a.e. $t\in\mathbb R$ and since $f$ is continuous, the last identity improves to $|f(t)|=1$ for every $t\in\mathbb R$.
Now define
$$
      G : \C \to \C, \quad  G(z):=f(z)\overline{f(\overline z)}.
$$
Then $G$ is an entire function and for every real $t$ we have
$$
        G(t)=f(t)\overline{f(t)}=|f(t)|^2=1.
$$
By the identity theorem for holomorphic functions, we have $G(z)= 1$ for every $z \in \C$.
In particular, $f$ has no zeros in $\mathbb C$.
We may therefore apply Lemma~\ref{lem:zero-free-exp-type} to $f$. Hence
there exist $C \in \mathbb C\setminus\{0\}$ and $a\in\mathbb C$ such that
$$
        f(z)=C e^{az},
        \quad z\in\mathbb C .
$$
Using $|f(t)|=1$ for every real $t$, we get
$$
        1=|f(t)|=|C|e^{\operatorname{Re}(a) t},
        \quad t\in\mathbb R .
$$
Taking $t=0$ gives $|C|=1$. Therefore,
$$
        e^{\operatorname{Re}(a)t}=1,
        \quad t\in\mathbb R,
$$
which forces $\operatorname{Re}(a)=0$. Hence $a=i\alpha$ for some
$\alpha\in\mathbb R$. Writing $\alpha=2\pi\lambda$ with $\lambda\in\mathbb R$ and setting $c=C$, we obtain
$$
        f(t)=ce^{2\pi i\lambda t},
        \quad t\in\mathbb R,
$$
with $|c|=1$. Since $F=f$ almost everywhere on $\mathbb R$, it follows
that
$
        F(t)=ce^{2\pi i\lambda t}
$
for almost every $t\in\mathbb R$. Clearly we have $\lambda \in \supp(\ft F)$.
\end{proof}

\subsection{Proof of the Han-Wang conjecture}

Up to this point, all results in this section have been established for an arbitrary $\alpha \in \R$; in particular, no irrationality assumption has been imposed. To prove the Han–Wang conjecture, we now combine the preceding results with the following observation, in which the irrationality of $\alpha$ plays a crucial role.

\begin{lemma}\label{lem:irrational-torus-obstructions}
Assume $\alpha \in \R \setminus \Q$.
\begin{enumerate}
\item If $f\in L^2(\mathbb T)$ satisfies $f(\theta)=f(\theta-\alpha)$ almost everywhere, then $f$ is constant almost everywhere.

\item If $c\in L^2(\mathbb T)$ satisfies
\begin{equation}\label{eq:c_relation}
c(\theta+\alpha)=\kappa e^{-2\pi i\theta}c(\theta)
\end{equation}
almost everywhere for some $|\kappa|=1$, then $c=0$.
\end{enumerate}
\end{lemma}

\begin{proof}
The first part of the statement is well-known, see for example the proof of \cite[Proposition 2.16]{EW12}.
For the second statement, let $c \in L^2(\T)$ and denote by $c_m$ its $m$-th Fourier coefficient,
$$
c_m=\int_0^1 c(x)e^{-2\pi i m x}\,dx .
$$
Then the equality in \eqref{eq:c_relation} implies that the Fourier coefficients of $c$ satisfy the relation
$$
e^{2\pi i m\alpha}c_m=\kappa c_{m+1},
\quad m\in\mathbb Z.
$$
Since $|\kappa|=1$, it follows that
$$
|c_{m+1}|=|c_m|,
\quad m\in\mathbb Z.
$$
Thus, if one Fourier coefficient $c_m$ were nonzero, then all Fourier
coefficients would have the same positive modulus, contradicting
$\{ c_m \}_{m\in\mathbb Z}\in\ell^2(\mathbb Z)$. Hence all $c_m$ vanish and the uniqueness theorem for Fourier coefficients implies $c=0$.
\end{proof}

In the next theorem we prove the Han-Wang conjecture. 

\begin{theorem}\label{thm:no-compactly-supported-onb}
If $\alpha$ is irrational then there exists no compactly supported $g \in L^2(\R)$ such that $\left\{M_nT_{m+\alpha n}g:m,n\in\mathbb Z\right\}$ is an orthonormal basis for $L^2(\R)$.
\end{theorem}

\begin{proof}
Suppose by contradiction that $\left\{M_nT_{m+\alpha n}g:m,n\in\mathbb Z\right\}$ is an orthonormal basis for some $g \in \lt$ with compact support. Let $K = \supp(g)$.

By Lemma~\ref{lem:fibre-compact-spectrum}, for almost every $\theta$, the
function $\omega\mapsto D(\omega,\theta)$ has distributional Fourier transform
supported in the compact set $-K$. Lemma~\ref{lem:bounded-unimodular-compact-spectrum}
therefore implies that, for almost every $\theta$, there exist
$$
c(\theta)\in\C,
\quad
\lambda(\theta)\in -K,
$$
with $|c(\theta)|=1$, such that
\begin{equation}\label{eq:D-exponential-representation}
D(\omega,\theta)
=
c(\theta)e^{2\pi i\lambda(\theta)\omega}
\end{equation}
for almost every $\omega\in\R$. This gives two functions $c(\theta)$ and $\lambda(\theta)$ which are defined for almost every $\theta \in \T$.

We show that $c(\theta)$ and $\lambda(\theta)$ are measurable, considered as functions in $L^\infty(\T)$. Since $-K$ is compact, choose $\tau>0$ so
small that the map
$$
x \mapsto e^{2\pi i\tau x}
$$
is injective on $-K$. From \eqref{eq:D-exponential-representation} we obtain the relation
$$
\int_0^1D(\omega+\tau,\theta)\overline{D(\omega,\theta)}\,d\omega
=
e^{2\pi i\tau\lambda(\theta)}.
$$
The left-hand side is measurable in $\theta$. Since the exponential map above
is injective on $-K$ and since $\lambda(\theta) \in -K$ for almost every $\theta \in \T$, it follows that $\lambda(\theta)$ is measurable. Then
$c(\theta)$ is measurable as well, since it can be represented in the form
$$
c(\theta)=\int_0^1D(\omega,\theta)e^{-2\pi i\lambda(\theta)\omega}\,d\omega.
$$
Now use the relation \eqref{eq:direct-fibre-covariance} which says that
$$
D(\omega+1,\theta)
=
e^{2\pi i(\alpha-\theta)}D(\omega,\theta-\alpha).
$$
Substituting \eqref{eq:D-exponential-representation}, and restricting to a
full-measure set of $\theta$'s on which both $\theta$ and $\theta-\alpha$ have
the representation \eqref{eq:D-exponential-representation}, gives
\begin{equation}\label{eq:c_lambda_relation}
c(\theta)e^{2\pi i\lambda(\theta)(\omega+1)}
=
e^{2\pi i(\alpha-\theta)}
c(\theta-\alpha)e^{2\pi i\lambda(\theta-\alpha)\omega}
\end{equation}
for almost every $\omega$. Define $A_\theta$ and $B_\theta$ via
$$
A_\theta = c(\theta)e^{2\pi i\lambda(\theta)}, \quad B_\theta = e^{2\pi i(\alpha-\theta)}c(\theta-\alpha).
$$
Since $|c(\theta)| = |c(\theta-\alpha)| = 1$, we have $|A_\theta| = |B_\theta| = 1$ and
$$
\frac{A_\theta}{B_\theta} e^{2\pi i (\lambda(\theta) -\lambda(\theta-\alpha) ) \omega} = 1
$$
for a.e. $\omega \in \R$. Since complex exponentials are linearly independent, it follows that
$$
\lambda(\theta) = \lambda(\theta-\alpha)
$$
and therefore $\lambda(\theta)=\lambda(\theta-\alpha)$ for a.e. $\theta \in \T$. Since $\lambda$ is a measurable and bounded function, it follows from Lemma \ref{lem:irrational-torus-obstructions}(i) that $\lambda$ is a constant function. Write $\lambda(\theta)=\lambda_0$ for almost every $\theta$. Then \eqref{eq:c_lambda_relation} gives
$$
c(\theta)e^{2\pi i\lambda_0}
=
e^{2\pi i(\alpha-\theta)}c(\theta-\alpha).
$$
Replacing $\theta$ by $\theta+\alpha$, we obtain
$$
c(\theta+\alpha)
=
 \kappa e^{-2\pi i\theta}c(\theta), \quad \kappa = e^{-2\pi i\lambda_0}.
$$
By Lemma~\ref{lem:irrational-torus-obstructions}(ii), the only $L^2$-solution
is $c=0$. This contradicts $|c(\theta)|=1$ almost everywhere.
\end{proof}

\section{Proof of Theorem \ref{thm:main1} and Theorem \ref{thm:main2}}

We begin with a structural invariance property that will be used repeatedly throughout this section (see, for instance, \cite{folland}).

\begin{lemma}\label{lma:chirp_invariance}
    Let the matrix $S \in \R^{2 \times 2}$ be defined by
    $$
    S = \begin{pmatrix}
        \mu & 0 \\ \nu & \mu^{-1}
    \end{pmatrix}, \quad \mu,\nu \in \R, \quad \mu \neq 0.
    $$
    Moreover, let $U_S : \lt \to \lt$ be the unitary operator
    $$
    (U_Sf)(t) = |\mu|^{-\frac12} e^{\pi i (\nu/\mu) t^2} f(t/\mu).
    $$
    Then $\mathbf{G}(g,\Lambda)$ is an orthonormal basis if and only if $\mathbf{G}(U_Sg,S\Lambda)$ is an orthonormal basis where $S\Lambda := \{ S\lambda : \lambda \in \Lambda \}$.
\end{lemma}

The previous lemma, in combination with the solution to the Han-Wang conjecture, gives the following statement.

\begin{theorem}\label{thm:no_onb}
    Let $\Lambda \subset \R^2$ be a lattice with density $D(\Lambda)=1$. If $\pi_1(\Lambda)$ is not discrete then there exists no compactly supported $g \in \lt$ such that $\mathbf{G}(g,\Lambda)$ is an orthonormal basis for $\lt$.
\end{theorem}
\begin{proof}
    Since $D(\Lambda)=1$, there exists $A \in \R^{2\times 2}$ with $\det A=1$ such that $\Lambda = A\Z^2$. Suppose that $A$ is given by
    $$
    A = \begin{pmatrix}
        a & b \\ c & d
    \end{pmatrix}.
    $$
    The projection of $\Lambda$ onto the first coordinate is given by $\pi_1(\Lambda) = a \Z + b \Z$. This subgroup of $\R$ is not discrete if and only if $a \neq 0$ and $\alpha \coloneqq \frac{b}{a} \not\in \Q$. With this choice of $\alpha$ we have the factorization
    $$
    A = \begin{pmatrix}
        a & 0 \\ c & a^{-1}
    \end{pmatrix}
    \begin{pmatrix}
        1 & \alpha \\ 0 & 1
    \end{pmatrix}.
    $$
    Defining
    $$
    L = \begin{pmatrix}
        a & 0 \\ c & a^{-1}
    \end{pmatrix}
    $$
    it follows that $\Lambda = L \Lambda_\alpha$ where $\Lambda_\alpha$ is given as in equation \eqref{eq:han_wang_matrix}. Now suppose towards a contradiction that there exists a compactly supported function $g \in \lt$ such that $\mathbf{G}(g,\Lambda)$ is an orthonormal basis for $\lt$. Define $S := L^{-1}$. Then $S$ is given by
    $$
    S = \begin{pmatrix}
        a^{-1} & 0 \\ -c & a
    \end{pmatrix}.
    $$
    Applying Lemma \ref{lma:chirp_invariance} with $\mu = a^{-1}$ and $\nu = -c$ it follows that $\mathbf{G}(U_S g, SL\Lambda_\alpha) = \mathbf{G}(U_S g, \Lambda_\alpha)$ is an orthonormal basis. Since $U_S g$ is compactly supported if and only if $g$ is compactly supported, we obtain a contradiction to Theorem \ref{thm:no-compactly-supported-onb}.
\end{proof}

Having ruled out the non-discrete case, we now consider lattices for which $\pi_1(\Lambda)$ is discrete. The next lemma shows that such lattices can always be represented, after a suitable change of basis, by a lower triangular matrix.

\begin{lemma}\label{lma:discrete_impl_lower_triangular}
    Let $\Lambda \subset \R^2$ be a lattice with density $D(\Lambda) = 1$. If $\pi_1(\Lambda)$ is discrete then there exists a lower triangular matrix $L \in \R^{2 \times 2}$ of the form
    $$
    L=\begin{pmatrix} a&0\\ c&a^{-1}\end{pmatrix},
\quad a>0,
    $$
    so that $\Lambda$ can be represented by $\Lambda = L\Z^2$.
\end{lemma}
\begin{proof}
    Define $H = \pi_1(\Lambda)$. Since $\Lambda$ is an additive subgroup of $\R^2$, it follows that $\pi_1(\Lambda)$ is a discrete additive subgroup of $\R$. Since every discrete additive subgroup of $\R$ that contains zero is of the form $\tau \Z$, it follows that $\pi_1(\Lambda) = \tau \Z$ for some $\tau > 0$.
    Now assume that $\Lambda = A \Z^2$ where
    $$
    A = \begin{pmatrix}
        \alpha & \beta \\ \gamma & \delta
    \end{pmatrix}.
    $$
    Since $D(\Lambda)=1$, we may choose the lattice basis matrix $A$ so that $\det A=1$.
    Since $\pi_1(\Lambda) = \alpha\Z + \beta\Z = \tau \Z$, there exist integers $m,n \in \Z$ such that
    $$
    \alpha = \tau m, \quad \beta = \tau n.
    $$
    We therefore have
    $
    \alpha\Z + \beta\Z = \tau ( m\Z + n \Z ) = \tau \,  \mathrm{gcd}(m,n) \Z = \tau \Z
    $
    which implies that $\mathrm{gcd}(m,n)=1$. Since every integer can be written as a linear combination of two coprime integers, it follows that there exist $r,s \in \Z$ such that
    $
    mr+ns = 1.
    $
    Now define the matrix $U \in \Z^{2 \times 2}$ by
    $$
    U = \begin{pmatrix}
        r & -n \\ s & m
    \end{pmatrix}.
    $$
    Then $\det U = 1$ and therefore $U$ leaves $\Z^2$ invariant, i.e., $U\Z^2 = \Z^2$. A direct calculation shows that
    $$
    AU =\begin{pmatrix}
        \tau & 0 \\ \gamma r + \delta s & -\gamma n + \delta m
    \end{pmatrix}.
    $$
    Since $\det(AU)=\det A\det U=1$ and the upper-left entry is $\tau>0$,
we must have
\[
-\gamma n+\delta m=\tau^{-1}.
\]
Thus $L =AU$ has the required form with $a=\tau$ and
$c=\gamma r+\delta s$.
\end{proof}

To complete the proof of Theorem \ref{thm:main2}, we appeal to the following characterization due to Liu \cite{L01}.

\begin{theorem}\label{thm:liu}
    Let $g \in \lt$ be compactly supported. Then the following are equivalent:
    \begin{enumerate}
        \item $\mathbf{G}(g,\Z^2)$ is an orthonormal basis
        \item $|g|=\1_E$ for a bounded measurable set $E \subset \R$ that tiles the real line with integer translates.
    \end{enumerate}
\end{theorem}
\begin{proof}[Proof of Theorem \ref{thm:main2}]
    Let $\mathbf{G}(g,\Lambda)$ be an orthonormal basis. By Lemma \ref{lma:discrete_impl_lower_triangular}, there exists a matrix
    $$
    L = \begin{pmatrix}
        a & 0 \\ c & a^{-1}
    \end{pmatrix}
    $$
    such that $\Lambda = L \Z^2$. Let $S = L^{-1}$, i.e.,
    $$
    S = \begin{pmatrix}
        a^{-1} & 0 \\ -c & a
    \end{pmatrix}.
    $$
    By Lemma \ref{lma:chirp_invariance} with $\mu = a^{-1}$ and $\nu = -c$ we have that $\mathbf{G}(g,\Lambda)$ is an orthonormal basis if and only if $\mathbf{G}(U_Sg,\Z^2)$ is an orthonormal basis. By Theorem \ref{thm:liu}, the latter holds if and only if $|U_S g|=\1_E$ for a bounded measurable set $E \subset \R$ that tiles the real line with integer translates. We have that
    $$
    |U_S g(t)| = |\mu|^{-\frac12} |g(t/\mu)| = |a|^{\frac12} |g(at)|
    $$
    Hence, the condition that $|U_S g|=\1_E$ for a bounded measurable set $E \subset \R$ that tiles the real line with integer translates is equivalent to the condition that
    $
    |g| = \frac{1}{\sqrt{a}} \1_\Omega
    $
    where $\Omega \subset \R$ is a bounded measurable set that tiles the real line by $a\Z$-translates.
\end{proof}

Theorem \ref{thm:main1} follows at once from the results in this section.

\begin{proof}[Proof of Theorem \ref{thm:main1}]
    Suppose that $\Lambda$ is a lattice and $g$ is a compactly supported function such that $\mathbf{G}(g,\Lambda)$ is an orthonormal basis. It is well-known that in this case the density of $\Lambda$ satisfies $D(\Lambda)=1$ (see for example \cite{IM18}). Moreover, $\pi_1(\Lambda)$ must be discrete by Theorem \ref{thm:no_onb}. Since $\pi_1(\Lambda)$ is a discrete subgroup containing zero, we have $\Lambda = a \Z$ for some $a>0$.

    On the other hand, if $D(\Lambda)=1$ and  $\pi_1(\Lambda)=a\Z$ for some $a>0$, then Theorem \ref{thm:main2} implies that the choice $g = \frac{1}{\sqrt{a}} \1_{[0,a]}$ gives an orthonormal basis $\mathbf{G}(g,\Lambda)$.
\end{proof}

\section{Proof of Theorem \ref{thm:main3}}

\subsection{Tilings of the unit cube}
We begin with a result of Gabardo, Lai, and Wang characterizing Gabor orthonormal bases generated by characteristic functions of unit intervals \cite[Theorem 1.2]{GLW15}. This theorem relates the orthonormal basis property to a geometric tiling condition in the time-frequency plane. For sets $S,T \subseteq \R^2$, we say that $S+T$ is a tiling if 
$$
\bigcup_{t \in T} (S+t) = \R^2
$$
and $(S+t) \cap (S+t')$ has Lebesgue measure zero for all $t,t' \in T$ with $t \neq t'$.

\begin{theorem}\label{thm:1}
    Let $\Lambda \subseteq \R^2$. Then $\mathbf{G}(\1_{[0,1]},\Lambda)$ is an orthonormal basis for $\lt$ if and only if $\mathbf{G}(\1_{[0,1]},\Lambda)$ is an orthogonal set and $[0,1]^2+\Lambda$ is a tiling of $\R^2$.
\end{theorem}

The second theorem provides a structural description of all tilings of $\R^2$ by translates of the unit cube, see \cite[Proposition 3.2]{GLW15}.

\begin{theorem}\label{thm:2}
    Let $\Lambda \subseteq \R^2$. If $[0,1]^2+\Lambda$ is a tiling of $\R^2$, then there exist $z \in \R^2$ and a sequence $\{ s_k \}_{k \in \Z} \subseteq [0,1)$ with $s_0=0$ such that $\Lambda \in \{ \Lambda_1, \Lambda_2 \}$ where
    $$
    \Lambda_1 = z + \bigcup_{k \in \Z} (\Z + s_k) \times \{ k \}, \quad \Lambda_2 = z + \bigcup_{k \in \Z} \{ k \} \times (\Z + s_k).
    $$
\end{theorem}

\subsection{Rotations of the time-frequency plane}

We next discuss the behavior of Gabor systems under rotations of the time-frequency plane, where rotations are represented by the matrices
\begin{equation}\label{def:rotation_matrix}
    R_\theta = \begin{pmatrix}
\cos \theta & -\sin \theta\\
\sin\theta & \cos\theta
\end{pmatrix}, \quad \theta \in \R.
\end{equation}
The corresponding action of a rotation on functions is implemented by the fractional Fourier transform. We briefly recall its definition and basic properties.

Let $\{ h_n \}_{n \in \N_0} \subset \lt$ be the normalized Hermite basis. The fractional Fourier transform of angle \(\theta\) is the unitary operator defined by
\begin{equation}\label{eq:frac_ft}
\ft_\theta : \lt \to \lt, \quad    \mathcal F_\theta f = \sum_{n=0}^{\infty} e^{-in\theta}\langle f,h_n\rangle h_n,
\end{equation}
where the series converges in \(L^2(\mathbb R)\). The operator $\ft_\theta$ generalizes the ordinary Fourier transform: if $\theta = \frac \pi 2$, then \eqref{eq:frac_ft} is the spectral decomposition of the Fourier operator and therefore $\ft_{\frac \pi 2} = \ft$.
Moreover, we have
\[
\mathcal F_0=I,
\quad
\mathcal F_\pi = \mathcal{R},
\quad
\mathcal F_{\theta+\phi}
=
\mathcal F_\theta\mathcal F_\phi,
\]
where $\theta, \phi \in \R$ and $\mathcal{R}f(x) = f(-x)$ denotes the reflection operator.
The following lemma expresses the invariance of the orthonormal basis property under simultaneous rotation of the lattice and application of the fractional Fourier transform. For a proof of this statement we refer to \cite{folland} (see also \cite[Chapter 9]{Groechenig}).

\begin{lemma}\label{lma:frac_invariance}
Let $g \in \lt$, let $\Lambda \subset \R^2$ and let $\theta \in \R$. Then $\mathbf G(g,\Lambda)$ is an orthonormal basis if and only if $\mathbf G(\mathcal F_\theta g,R_{-\theta}\Lambda)$ is an orthonormal basis.
\end{lemma}

\subsection{The Iosevich-Mayeli problem}

For a set $S \subseteq \R^2$, we define its upper Beurling density via
$$
D^+(S)
=
\limsup_{r\to\infty}
\sup_{x\in\mathbb R^2}
\frac{\#\bigl(S \cap \mathbb{B}(x,r)\bigr)}{|\mathbb{B}(x,r)|},
$$
where $\mathbb{B}(x,r) = \{ y \in \R^2 : |x-y| \leq r \}$ denotes the Euclidean ball of radius $r \geq 0$ around $x \in \R^2$.

\begin{theorem}\label{thm:io_answer}
    Let $\theta \in \R$ such that $\theta \not \in \frac \pi 2 \Z$. Further, define $g:=\ft_{-\theta} \1_{[0,1]}$ and $\Lambda := R_{\theta}\Z^2$. Then the following holds:
    \begin{enumerate}
        \item $\mathbf{G}(g,\Lambda)$ is an orthonormal basis for $\lt$,
        \item $\mathbf{G}(g,A \times B)$ is not an orthonormal basis for any $A,B \subseteq \R$.
    \end{enumerate}
\end{theorem}
\begin{proof}
Since $\mathbf{G}(\1_{[0,1]},\Z^2)$ is an orthonormal basis, it follows from Lemma \ref{lma:frac_invariance}, that $\mathbf{G}(\ft_{-\theta} \1_{[0,1]},R_{\theta}\Z^2) = \mathbf{G}(g,\Lambda)$ is an orthonormal basis.

Let $A,B \subseteq \R$ be two arbitrary subsets of the real line. We need to show that $\mathbf{G}(g, A \times B)$ is not an orthonormal basis for $\lt$.
Assume the contrary, i.e., $\mathbf{G}(g, A \times B)$ is an orthonormal basis.
Applying the fractional Fourier transform of order $\theta$, and using Lemma \ref{lma:frac_invariance}, we obtain that
    $
    \mathbf{G}(\1_{[0,1]}, R_{-\theta}(A \times B))
    $
    is an orthonormal basis for $\lt$.
    
According to Theorem \ref{thm:1}, $[0,1]^2 + R_{-\theta}(A \times B)$ must be a tiling for $\R^2$. Theorem \ref{thm:2} implies that $R_{-\theta}(A \times B)$ is of the form $\Lambda_1$ or $\Lambda_2$. Suppose that it is of the form $\Lambda_1$ (the case where it is of the form $\Lambda_2$ can be treated in a similar way), i.e.,
\begin{equation}\label{eq:set_eq}
    R_{-\theta}(A \times B) = z + \bigcup_{k \in \Z} (\Z + s_k) \times \{ k \}
\end{equation}
where $z = (z_1,z_2) \in \R^2$ and $\{ s_k \}_{k \in \Z} \subseteq [0,1)$ with $s_0=0$. It follows that for every $a \in A$ and $b \in B$ we have
$$
-a \sin \theta + b \cos \theta \in z_2 + \Z.
$$
Fixing some $b_0 \in B$ we obtain from the previous relation that the set $A$ satisfies the inclusion
$$
A \subseteq \frac{-z_2 + b_0 \cos \theta}{\sin \theta} + \frac{1}{\sin \theta} \Z.
$$
Here we used the assumption $\theta \not \in \frac \pi2 \Z$ which gives $\sin \theta \neq 0$ and $\cos \theta \neq 0$. Hence, $A$ is contained in an arithmetic progression with spacing $\frac{1}{|\sin \theta|}$. An analogous argument shows that $B$ is contained in an arithmetic progression with spacing $\frac{1}{|\cos \theta|}$. From this, we obtain that the upper Beurling density of $A \times B$ satisfies the bound
$$
D^+(A \times B) \leq |\sin \theta \cdot \cos \theta| \leq \frac 12.
$$
Since the upper Beurling density is rotation invariant, we obtain $D^+(R_{-\theta}(A \times B)) \leq \frac 12$. But $D^+(\Lambda_1) = 1$, which gives a contradiction to the set-equality \eqref{eq:set_eq}. Therefore, $\mathbf{G}(g,A \times B)$ cannot be an orthonormal basis. Since $A,B$ were arbitrary, the result follows.
\end{proof}

According to the discussion in the introduction, the function $g$ in Theorem \ref{thm:io_answer} cannot be compactly supported. This can be seen directly, when working with the integral representation of $\ft_\theta$, which says that if $\theta \not \in \frac \pi2 \Z$ then for every $f \in L^1(\R) \cap L^2(\R)$ one has
\begin{equation}\label{eq:int_form}
    \ft_\theta f(\xi) = \gamma_\theta |\sin\theta|^{-1/2} \int_{\mathbb R} e^{\pi i \left(\frac{\xi^2+y^2}{\tan\theta}-\frac{2\xi y}{\sin\theta} \right)} f(y)\,dy.
\end{equation}
This integral formula and related representations can be found, for instance, in Folland's book \cite{folland}.
From \eqref{eq:int_form}, we see that $\ft_\theta f(\xi)$ is, up to a dilation and multiplication with a non-vanishing function, the Fourier transform of the compactly supported function $$y \mapsto \1_{[0,1]}(y) e^{ \frac{\pi i y^2}{\tan \theta} }$$ which arises as a product of $\1_{[0,1]}$ and a quadratic exponential. Hence, $\ft_\theta \1_{[0,1]}$ has unbounded support whenever $\theta \not \in \frac \pi2 \Z$.

\section*{Acknowledgments}

The author thanks Nir Lev for very helpful discussions on this project and Jordy Timo van Velthoven for helpful comments on an earlier version of the manuscript.

\medskip
The author is grateful to the Azrieli Foundation for the award of an Azrieli Fellowship and acknowledges the support of this research by ISF Grant No.~854/25 and ISF Grant No.~1044/21.

\bibliographystyle{plain}
\bibliography{bibfile}

\begin{thebibliography}{10}

\bibitem{AAK20}
E.~Agora, J.~Antezana, and M.~N. Kolountzakis.
\newblock {Tiling functions and Gabor orthonormal basis}.
\newblock {\em Appl. Comput. Harmon. Anal.}, 48(1):96--122, 2020.

\bibitem{BM21}
B.~Behera and N.~Molla.
\newblock {Characterization of Schauder basis property of Gabor systems in
  local fields}.
\newblock {\em Acta Sci. Math. (Szeged)}, 87:517--539, 2021.

\bibitem{bekka}
B.~Bekka.
\newblock {Square integrable representations, von Neumann algebras and an
  application to Gabor analysis}.
\newblock {\em J. Fourier Anal. Appl.}, 10(4):325--349, 2004.

\bibitem{BO18}
H.~Burgiel and V.~Oussa.
\newblock {Gabor orthonormal bases generated by indicator functions of
  parallelepiped-shaped sets}.
\newblock {\em Adv. Pure Appl. Math.}, 9(2):93--107, 2018.

\bibitem{CL18}
R.~Chung and C.-K. Lai.
\newblock {Non-symmetric convex polytopes and Gabor orthonormal bases}.
\newblock {\em Proc. Amer. Math. Soc.}, 146(12):5147--5155, 2018.

\bibitem{DL22}
A.~Debernardi~Pinos and N.~Lev.
\newblock {Gabor orthonormal bases, tiling and periodicity}.
\newblock {\em Math. Ann.}, 384:1461--1467, 2022.

\bibitem{EW12}
M.~Einsiedler and T.~Ward.
\newblock {\em {Ergodic Theory: with a view towards Number Theory}}.
\newblock Springer London, 2012.

\bibitem{folland}
G.~B. Folland.
\newblock {\em {Harmonic Analysis in Phase Space}}.
\newblock Princeton University Press, 2016.

\bibitem{fuglede1974commuting}
B.~Fuglede.
\newblock Commuting self-adjoint partial differential operators and a group
  theoretic problem.
\newblock {\em J. Funct. Anal.}, 16(1):101--121, 1974.

\bibitem{GLW15}
J.-P. Gabardo, C.-K. Lai, and Y.~Wang.
\newblock {Gabor orthonormal bases generated by the unit cubes}.
\newblock {\em J. Funct. Anal.}, 269(5):1515--1538, 2015.

\bibitem{greenfeld2017fuglede}
R.~Greenfeld and N.~Lev.
\newblock Fuglede’s spectral set conjecture for convex polytopes.
\newblock {\em Anal. PDE}, 10(6):1497--1538, 2017.

\bibitem{grepstad2026bounded}
Sigrid Grepstad and Mihail~N Kolountzakis.
\newblock Bounded common fundamental domains for two lattices.
\newblock {\em Adv. Math.}, 487:110776, 2026.

\bibitem{Groechenig}
K.~Gröchenig.
\newblock {\em {Foundations of Time-Frequency Analysis}}.
\newblock Birkhäuser Basel, 2001.

\bibitem{HW01}
D.~Han and Y.~Wang.
\newblock {Lattice tiling and the Weyl-Heisenberg frames}.
\newblock {\em Geom. Funct. Anal.}, 11:742--758, 2001.

\bibitem{HW04}
D.~Han and Y.~Wang.
\newblock {The existence of Gabor bases and frames}.
\newblock {\em Contemp. Math.}, 345:183--192, 2004.

\bibitem{HP06}
C.~Heil and A.~M. Powell.
\newblock {Gabor Schauder bases and the Balian-Low theorem}.
\newblock {\em J. Math. Phys.}, 47(11), 2006.

\bibitem{H83}
L.~H{\"o}rmander.
\newblock {\em {The analysis of linear partial differential operators. I:
  Distribution theory and Fourier analysis}}, volume 256 of {\em Grundlehren
  der Mathematischen Wissenschaften}.
\newblock Springer Verlag, Berlin, 1983.

\bibitem{IKL21}
A.~Iosevich, M.~Kolountzakis, Y.~Lyubarskii, A.~Mayeli, and J.~Pakianathan.
\newblock {On Gabor orthonormal bases over finite prime fields}.
\newblock {\em Bull. Lond. Math. Soc.}, 53(2):380--391, 2021.

\bibitem{IM18}
A.~Iosevich and A.~Mayeli.
\newblock {Gabor Orthogonal Bases and Convexity}.
\newblock {\em Discrete Anal.}, 2018.

\bibitem{LM19}
C.-K. Lai and A.~Mayeli.
\newblock {Non-separable Lattices, Gabor Orthonormal Bases and Tilings}.
\newblock {\em J. Fourier Anal. Appl.}, 25:3075--3103, 2019.

\bibitem{lev2022fuglede}
N.~Lev and M.~Matolcsi.
\newblock {The Fuglede conjecture for convex domains is true in all
  dimensions}.
\newblock {\em Acta Math.}, 228(2):385--420, 2022.

\bibitem{LT26}
N.~Lev and A.~Tselishchev.
\newblock {Gabor unconditional bases and frames in $L^p(\mathbb{R})$}.
\newblock {\em arXiv:2605.17970}, 2026.

\bibitem{L05}
Y.-Z. Li.
\newblock {A note on Gabor orthonormal bases}.
\newblock {\em Proc. Amer. Math. Soc.}, 133(8):2419--2428, 2005.

\bibitem{L01}
Y.~Liu.
\newblock {A characterization for windowed Fourier orthonormal basis with
  compact support}.
\newblock {\em Acta Math. Sin. (Engl. Ser.)}, 17:501--506, 2001.

\bibitem{LW03}
Y.~Liu and Y.~Wang.
\newblock {The uniformity of non-uniform Gabor bases}.
\newblock {\em Adv. Comput. Math.}, 18:345--355, 2003.

\bibitem{N24}
F.~Nicola.
\newblock {Maximally localized Gabor orthonormal bases on locally compact
  Abelian groups}.
\newblock {\em Adv. Math.}, 451, 2024.

\bibitem{tao2004fuglede}
T.~Tao.
\newblock Fuglede's conjecture is false in 5 and higher dimensions.
\newblock {\em Math. Res. Lett.}, 11(2):251--258, 2004.

\bibitem{Z21}
W.~Zhou.
\newblock {On the construction of discrete orthonormal Gabor bases on finite
  dimensional spaces}.
\newblock {\em Appl. Comput. Harmon. Anal.}, 55:270--281, 2021.

\end{thebibliography}

\end{document}